\documentclass[12pt]{amsart}
\usepackage[all]{xy}
\usepackage{comment}
\usepackage{color}
\usepackage{graphicx}
\def\NZQ{\mathbb}               % the font for N,Z,Q,R,C

\def\ZZ{{\NZQ Z}}
\def\RR{{\NZQ R}}

\def\frk{\mathfrak}               % font for "Fraktur"

\def\Phi{{\frk N}}
\def\eb{{\bold e}}

\def\opn#1#2{\def#1{\operatorname{#2}}} % to make operators
\opn\chara{char} \opn\length{\ell} \opn\pd{pd} \opn\rk{rk}
\opn\projdim{proj\,dim} \opn\injdim{inj\,dim} \opn\rank{rank}
\opn\depth{depth} \opn\grade{grade} \opn\height{height}
\opn\embdim{emb\,dim} \opn\codim{codim}

\opn\Tr{Tr} \opn\bigrank{big\,rank}
\opn\superheight{superheight}\opn\lcm{lcm}
\opn\trdeg{tr\,deg}%\emph{
\opn\reg{reg} \opn\lreg{lreg} \opn\ini{in} \opn\lpd{lpd}
\opn\size{size}\opn{\mult}{mult}
\opn\div{div} \opn\Div{Div} \opn\cl{cl} \opn\Cl{Cl}
\opn\Spec{Spec} \opn\Supp{Supp} \opn\supp{supp} \opn\Sing{Sing}
\opn\Ass{Ass} \opn\Min{Min}
\opn\Ann{Ann} \opn\Rad{Rad} \opn\Soc{Soc}
\opn\Syz{Syz} \opn\Im{Im} \opn\Ker{Ker} \opn\Coker{Coker}
\opn\Am{Am} \opn\Hom{Hom} \opn\Tor{Tor} \opn\Ext{Ext}
\opn\End{End} \opn\Aut{Aut} \opn\id{id}

\opn\nat{nat}
\opn\pff{pf}%   \pf exists already
\opn\Pf{Pf} \opn\GL{GL} \opn\SL{SL} \opn\mod{mod} \opn\ord{ord}
\opn\Gin{Gin}
\opn\Hilb{Hilb}\opn\adeg{adeg}\opn\std{std}\opn\ip{infpt}
\opn\Pol{Pol}
\opn\sat{sat}
\opn\Var{Var}
\opn\Gen{Gen}
\opn\vol{vol}

\opn\convy{co}
\opn\aff{aff} \opn\con{conv} \opn\relint{relint} \opn\st{st}
\opn\lk{lk} \opn\cn{cn} \opn\core{core} \opn\vol{vol}
\opn\link{link} \opn\star{star}
\opn\gr{gr}

\def\Pc{{\mathcal P}}

\def\Hc{{\mathcal H}}
\def\pot#1#2{#1[\kern-0.28ex[#2]\kern-0.28ex]}

\opn\dirlim{\underrightarrow{\lim}}
\opn\inivlim{\underleftarrow{\lim}}
\let\to=\rightarrow

\def\Implies{\ifmmode\Longrightarrow \else
        \unskip${}\Longrightarrow{}$\ignorespaces\fi}
\def\implies{\ifmmode\Rightarrow \else
        \unskip${}\Rightarrow{}$\ignorespaces\fi}
\def\iff{\ifmmode\Longleftrightarrow \else
        \unskip${}\Longleftrightarrow{}$\ignorespaces\fi}

\let\:=\colon
\newtheorem{Theorem}{Theorem}[section]
\newtheorem{Lemma}[Theorem]{Lemma}
\newtheorem{Corollary}[Theorem]{Corollary}
\newtheorem{Proposition}[Theorem]{Proposition}

\newtheorem{Example}[Theorem]{Example}

\newtheorem{Algorithm}[Theorem]{Algorithm}
\newtheorem{Question}[Theorem]{Question}
\let\epsilon\varepsilon
\let\phi=\varphi
\let\kappa=\varkappa
\opn\dis{dis}
\def\pnt{{\raise0.5mm\hbox{\large\bf.}}}

\opn\Lex{Lex}

\newcommand{\good}{quadratic}
\numberwithin{equation}{section}

\begin{document}
\title{Separating hyperplanes of edge polytopes}
\author{Takayuki Hibi, Nan Li and Yan X Zhang}
\address{Takayuki Hibi,
Department of Pure and Applied Mathematics,
Graduate School of Information Science and Technology,
Osaka University,
Toyonaka, Osaka 560-0043, Japan}
\email{hibi@math.sci.osaka-u.ac.jp}
\address{Nan Li,
Department of Mathematics,
Massachusetts Institute of Technology,
Cambridge, MA 02139, USA}
\email{nan@math.mit.edu}
\address{Yan X Zhang,
Department of Mathematics,
Massachusetts Institute of Technology,
Cambridge, MA 02139, USA}
\email{yanzhang@math.mit.edu}
\thanks{
{\bf 2010 Mathematics Subject Classification:}
Primary 00000; Secondary 00000. 
}
\begin{abstract}
Let $G$ be a finite connected simple graph with $d$ vertices
and let $\Pc_G \subset \RR^d$ be the edge polytope of $G$.
We call $\Pc_G$ \emph{decomposable} if $\Pc_G$ decomposes into integral polytopes $\Pc_{G^+}$ and $\Pc_{G^-}$ via a hyperplane. In this paper, we explore various aspects of decomposition of $\Pc_G$: we give an algorithm deciding the decomposability of $\Pc_G$, we prove that $\Pc_G$ is normal if and only if both $\Pc_{G^+}$ and $\Pc_{G^-}$ are normal, and we also study how a condition on the toric ideal of $\Pc_G$ (namely, the ideal being generated by quadratic binomials) behaves under decomposition.

\end{abstract}
\maketitle
\section*{introduction}
A {\em simple graph} is a graph with no loops and no multiple edges.
Let $G$ be a finite connected simple graph with vertex set
$[d] = \{ 1, \ldots, d \}$
and edge set $E(G) = \{ e_1, \ldots, e_n \}$.
Let $\eb_i$ be the $i$-th unit coordinate vector
of the euclidean space $\RR^d$.
If $e = (i, j)$ is an edge of $G$, then
we set
$\rho(e) = \eb_i + \eb_j \in \RR^d$.
The {\em edge polytope} $\Pc_G$ of $G$ is
the convex hull of
$\{ \rho(e_1), \ldots, \rho(e_n) \}$ in $\RR^d$.
The basics of edge polytopes are studied in
\cite{OHnormal}, \cite{OHregular}, \cite{OHkoszul},
\cite{OHquadratic} and \cite{OHcompressed}.

In this paper, we study the decompositions of edge polytopes via hyperplanes. Recall that a convex polytope is \emph{integral} if
all of its vertices have integral coordinates; in particular, $\Pc_G$ is an integral polytope. Let $\partial P$ denote the boundary of a polytope $P$.
We say that $P$ is {\em decomposable} if there exists a hyperplane
$\Hc$ of $\RR^d$ with
$\Hc \cap (P \setminus \partial P) \neq \emptyset$
such that each of the convex polytopes
$P \cap \Hc^{(+)}$ and
$P \cap \Hc^{(-)}$ is integral. Here $\Hc^{(+)}$ and $\Hc^{(-)}$ are the closed half-spaces of $\RR^d$
with $\Hc^{(+)} \cap \Hc^{(-)} = \Hc$.
Such a hyperplane $\Hc$ is called a {\em separating hyperplane} of $P$. 

We start with a nice fact, Lemma \ref{subpolytope}, which shows that when $\Pc_G$ is decomposable into $\Pc_{G^+}$ and $\Pc_{G^-}$, the two subpolytopes are again edge polytopes. This lemma sets the theme for our paper, namely that certain polytope properties (in this case, being an edge polytope) are well-behaved under decomposition when we restrict to edge polytopes. Furthermore, thanks to edge polytopes having associated graphs, these properties may correspond to easily-visualized combinatorial properties of the underlying graphs. When both of these conditions are in place, checking difficult properties on a decomposable polytope may then be reduced to checking those properties on the graphs corresponding to the smaller pieces in its decomposition. This is one motivation for the study of decomposable edge polytopes.

  After the basics, we look at the fundamental problem of determining which edge polytopes are decomposable. 
In Corollary~\ref{cycleoflength4}, we see that a necessary condition for the decomposability of $\Pc_G$ is
that $G$ possesses at least one cycle of length $4$. We also provide Algorithm \ref{base 2 algorithm}, which
decides the decomposability of an edge polytope. We then focus on the case when $G$ is a complete multipartite graph and count the number of separating hyperplanes for such $\Pc_G$ in Theorem \ref{commul}.

Then, we consider the property of normality under decomposition. The characterization of normal polytopes is one of the fundamental questions of the study of lattice polytopes that can also shed light on other properties. For example, it is hard to check combinatorially whether a toric ring of a polytope is Cohen-Macaulay, but Hochster's Theorem \cite{hoch} gives that this is implied by the normality of the polytope. This paper is part of a continuing effort to understand normality in edge polytopes that started in \cite{OHnormal} and \cite{SVV}.  It is known \cite[Corollary 2.3]{OHnormal}
that an edge polytope
$\Pc_G$ is normal
if and only if $G$ satisfies the so-called ``odd cycle condition''
(\cite{FHM}). We use this combinatorial criterion to show that normality of edge polytopes also behaves nicely under decomposition in Theorem \ref{Yan:Normality}; specifically, in a decomposition of an edge polytope $\Pc_G$, then $\Pc_G$ is normal if and only if both subpolytopes are normal.

Finally, we examine the connections between toric ideals and lattice polytopes. In the theory of toric ideals (\cite{Stu}),
special attention has been given to toric ideals
generated by quadratic binomials. As before, we study this property under decomposition by checking a combinatorial condition on the graph, in this case with the help of \cite{OHquadratic}.  We show in Theorem \ref{good if both are good} that $I_G$ is generated by quadratic binomials if both toric ideals corresponding to the subpolytopes have this property. However, we stress that the converse does not hold.

There are other similar questions to be asked following our work. For example, we can ask how the property of the toric ring possessing a quadratic (or squarefree) Gr\"{o}bner basis behaves under decomposition. Unfortunately, we currently do not have a combinatorial description for this property and thus do not know how to approach this problem.

\section{Separating Hyperplanes and Decompositions}

The vertices of the edge polytope $\Pc_G$ of $G$
are $\rho(e_1), \ldots, \rho(e_n)$, but not all edges of the form $(\rho(e_i), \rho(e_j))$ actually occur. For $i \neq j$, let $\convy(e_i,e_j)$ be the convex hull of the pair $\{\rho(e_i), \rho(e_j)\}$. The edges of $\Pc_G$ will be a subset of these $\convy(e_i, e_j)$. For edges $e = (i,j)$ and $f=(k,\ell )$, call the pair of edges $(e,f)$ \emph{cycle-compatible} with $C$ if there exists a $4$-cycle $C$ in the subgraph of $G$ induced by $\{i,j,k, \ell\}$ (in particular, this implies that $e$ and $f$ do not share any vertices). The following result allows us to identify the $\convy(e_i,e_j)$ that are actually edges of $\Pc_G$ using the notion of cycle-compatibility.

\begin{Lemma}[\cite{OHsimple}]
\label{edgeofedgepolytope}
Let $e$ and $f$ be edges of $G$ with
$e \neq f$.  Then $\convy(e,f)$ is an edge of $\Pc_G$ if and only if $e$ and $f$ are \emph{not} cycle-compatible.
\end{Lemma}

% \begin{Lemma}[\cite{OHsimple}]
% \label{edgeofedgepolytope}
% Let $e$ and $f$ be edges of $G$ with
% $e \neq f$.  Then $\convy(e,f)$ is an edge of $\Pc_G$ if and only if one of the following conditions
% are satisfied:
% \begin{itemize}
% \item
% $e \cap f \neq \emptyset$;
% \item
% $e \cap f = \emptyset$ and $(e,f)$ is \emph{not} cycle-compatible.
% \end{itemize}
% \end{Lemma}

\begin{Example}
\label{completegraph}
{\em
Let $K_d$ denote the complete graph on $[d]$ and $e$ and $f$ be edges of $K_d$. We have that $e \cap f = \emptyset$ exactly when $e$ and $f$ are cycle-compatible. We can then compute the number of edges of $\Pc_{K_d}$ by counting the $2$-element subsets $\{ e, f \}$ of
$E(K_d)$ with $e \cap f \neq \emptyset$, which is
\[
d{d - 1 \choose 2} = \frac{d^2(d-1)}{2}.
\]
}
\end{Example}

\begin{Question}
\label{question:maximal}
{\em
Fix integers $d \geq 2$.  What is the maximal number of possible edges of $\Pc_G$, as $G$ ranges among all finite connected simple graphs on $[d]$?
}
\end{Question}

% We say that $\Pc_G$ is {\em decomposable} if there exists a
% hyperplane $\Hc$ of $\RR^d$ with $\Hc \cap (\Pc_G \setminus \partial
% \Pc_G) \neq \emptyset$ such that each of the convex polytopes $\Pc_G
% \cap \Hc^{(+)}$ and $\Pc_G \cap \Hc^{(-)}$ are integral. Here
% $\Hc^{(+)}$ and $\Hc^{(-)}$ are the closed half-spaces of $\RR^d$
% defined by $\Hc$ with $\Hc^{(+)} \cap \Hc^{(-)} = \Hc$ and where
% $\partial \Pc_G$ is the boundary of $\Pc_G$. Such a hyperplane $\Hc$
% is called a {\em separating hyperplane} of $\Pc_G$.

The condition that $\Pc_G \cap \Hc^{(+)}$ and $\Pc_G \cap
\Hc^{(-)}$ are integral is equivalent to having no edge in $\Pc_G$ intersecting $\Hc$ except possibly at the endpoints. This is, by Lemma~\ref{edgeofedgepolytope}, equivalent to the the following condition: we say that $G$ satisfies the \emph{4-cycle condition} if for any pair of edges $e,f\in E(G)$ such that $\rho(e)\in \Hc^{(+)}\backslash \Hc$ and $\rho(f)\in \Hc^{(-)}\backslash \Hc$, $e$ and $f$ are cycle-compatible.

We can make a simplification of $\Hc$:
\begin{Lemma}
\label{planeshift}
If $\Pc_G$ is decomposable via a hyperplane $\Hc$ that does not go through the origin, there exists a hyperplane $\Hc'$ that gives the same decomposition, with the additional condition that $\Hc$ goes through the origin.
\end{Lemma}
\begin{proof}
Suppose $\Hc \subset \RR^d$ is defined by
$$a_1 x_1 + \cdots + a_d x_d = b,$$
with each $a_i, b \in \RR$.
Let $\Hc'$ denote the hyperplane
defined by the equation
$$(a_1 - b/2) x_1 + \cdots + (a_d - b / 2) x_d = 0.$$
Since $\Hc \cap \Pc_G$
lies in the hyperplane defined by the equation $x_1 + \cdots + x_d = 2$,
it follows that $\Hc \cap \Pc_G = \Hc' \cap \Pc_G$ and the decomposition is not affected.
\end{proof}

Lemma~\ref{planeshift} allows us to assume that all our separating hyperplanes for edge polytopes go through the origin; we will always make this assumption from now on. When we do, assume that $\Hc^{(+)}$ contains points $(x_1, \ldots, x_n)$ where $\sum a_i x_i \geq 0$ and $\Hc^{(-)}$ contains points where $\sum a_i x_i \leq 0$.

We now introduce the function
$s_\Hc : E(G) \to \{ 0, 1, -1 \}$
defined by setting $s_\Hc((i, j))$ to be the sign of $a_i + a_j$, allowing $0$.
The function $s_\Hc$ enables us to call an edge $e$ ``positive,'' ``negative,''
or ``zero,'' corresponding to whether the associated vertex $\rho(e)$ in $\Pc_G$ is in $\Hc^{(+)}\backslash \Hc$,  $\Hc^{(-)}\backslash \Hc$, or $\Hc$, respectively. We will repeatly use the following fact:

\begin{Corollary}
\label{opp}
For any positive edge $e$ and negative edge $f$ in a decomposition, $e$ and $f$ must be cycle-compatible in a cycle $(e,g,f,h)$, where $g$ and $h$ are zero edges.
\end{Corollary}
\begin{proof}
The cycle-compatibility is an immediate corollary of Lemma~\ref{edgeofedgepolytope}. This also implies we cannot have a positive edge sharing a vertex with a negative edge, so since the other two edges share a vertex with both $e$ and $f$, they must be zero edges.
\end{proof}

Since the hyperplane $\Hc$ decomposes $\Pc_G$, we must have at least one
positive edge and at least one negative edge. Thus, we also have:

\begin{Corollary}
\label{cycleoflength4}
Suppose that $\Pc_G$ is decomposable.  Then $G$ must possess at least one cycle of length $4$.
\end{Corollary}

\begin{Example}
{\em
The converse of Corollary~\ref{cycleoflength4} does not hold for all graphs. Let $G$ be the following graph. 

\begin{center}
\includegraphics{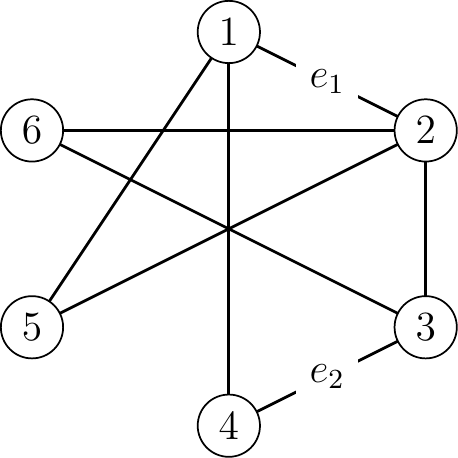}
\end{center}

Even though $G$ possesses a cycle of length $4$,
the edge polytope $\Pc_G$ is indecomposable. If $\Hc$ were a separating
hyperplane with equation $\sum_{i}^6 a_i x_i = 0$, since there is only one $4$-cycle and the two pairs of nonadjacent edges are symmetric, we may assume without loss of generality that $e_1$ is positive and $e_2$ is negative. Since any pair of positive and negative edges must form a $4$-cycle and there are no other $4$-cycles, all remaining $e_i$ must be zero edges. Therefore, we have
\[
a_2 + a_3 = a_1 + a_4 = a_1 + a_5 = a_2 + a_5 = a_2 + a_6 = a_3 + a_6 = 0.
\]
Thus $a_1 = a_2 = a_3 = - a_4 = - a_5 = - a_6$.  In particular $a_3 + a_4 = 0$.
However, since $\rho(e_2) \in \Hc^{(-)} \setminus \Hc$,
one has $a_3 + a_4 \neq 0$, a contradiction.
}
\end{Example}

The following important result tells us that being an edge polytope is hereditary under decomposition:

\begin{Lemma}
\label{subpolytope} Let $G$ be a finite connected simple graph on $[d]$ and suppose that $\Pc_G \subset \RR^d$ is decomposable by $\Hc$. Then each of the subpolytopes $\Pc_G \cap \Hc^{(+)}$ and $\Pc_G \cap \Hc^{(-)}$ is again an edge polytope. More precisely, one has
connected spanning subgraphs $G_+$ and $G_-$ with $\Pc_G \cap
\Hc^{(+)} = \Pc_{G_+}$ $\Pc_G \cap \Hc^{(-)} = \Pc_{G_-}$.
\end{Lemma}
%\com{this seems more ``fundamental.'' Why isn't this in the first
%section?(\textbf{nan}: can we leave this comment to HIBI?)}

\begin{proof}
Let $G_+$ and $G_-$ be subgraphs of $G$ with $E(G_+) = \{ e \in E(G)
: s_\Hc(e) \geq 0\}$ and $E(G_-) = \{ e \in E(G) : s_\Hc(e) \leq
0\}$. Let $\Pc_{G_+}$ be the subpolytope of $\Pc_G$ which is the
convex hull of $\{ \rho(e) : e \in E(G_+) \}$ and let $\Pc_{G_-}$ be the
subpolytope of $\Pc_G$ which is the convex hull of $\{ \rho(e) : e
\in E(G_-) \}$. One can assume $\Pc_G \cap \Hc^{(+)} = \Pc_{G_+}$
and $\Pc_G \cap \Hc^{(-)} = \Pc_{G_-}$.

Since the dimension of each of
the subpolytopes $\Pc_{G_+}$ and $\Pc_{G_-}$
coincides with that of $\Pc_G$, it follows from % the dimension formula
\cite[Proposition 1.3]{OHnormal} % of edge polytopes
that both subgraphs $G_+$ and $G_-$ must be connected spanning
subgraphs of $G$,
% In particular each of $\Pc_{G_+}$ and $\Pc_{G_-}$ is an edge polytope,
as desired.
\end{proof}

\begin{Example}
\label{ex:first}
{\em
An example of a decomposition of a graph $G$ using the hyperplane $-x_1+x_4-x_5+x_6=0$. For edges $(i,j)$ in $G$, $a_i+a_j$ is only nonzero for $\{i,j\} = \{1,2\}$ or $\{3,4\}$, with values $-1$ and $1$ respectively. These correspond to the two non-zero edges.  We give two equivalent graphical representations of this decomposition, one by showing $G_+$ and $G_-$ explicitly and one by marking the non-zero edges $+$ or $-$.

\begin{center}
\begin{tabular}{c|c|c}
$G$ &  $G_+$    &   $G_-$ \\
\hline
\includegraphics{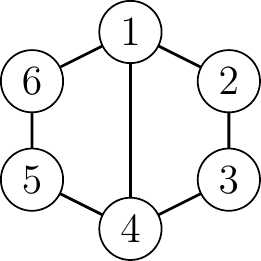}
&
\includegraphics{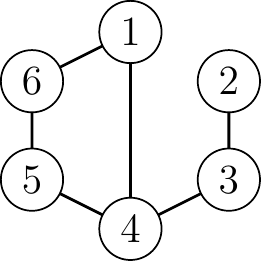}
&
\includegraphics{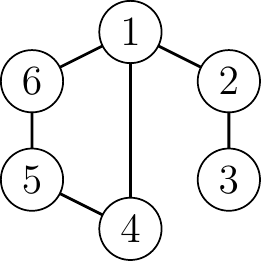}
\\
\hline
\end{tabular}
\end{center}
\begin{center}
\includegraphics{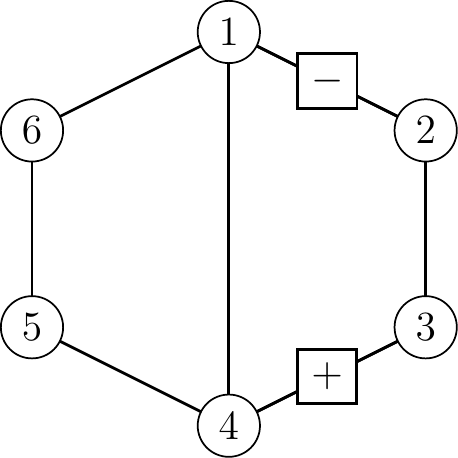}
\end{center}
}
\end{Example}

We will now prove a result that allows us to normalize the coefficients used in the hyperplane in a decomposition. Let the \emph{weight} of a vertex $i$ of $G$ be $a_i$ and let the \emph{signature} of the edge $(i,j)$ (or the corresponding vertex $\rho((i,j))$ in $\Pc_G$) be the set of weights $\{a_i, a_j\}$ and the \emph{weight} of $(i,j)$ be $a_i + a_j$. An edge $e$ has the sign function $s_\Hc(e)$ equal to the sign of the sum of the weights in its signature.

\begin{Proposition}
\label{Yan:hyperplane}
Suppose we have a decomposition $\Pc_G = \Pc_G^{(+)} \cup \Pc_G^{(-)}$. Then for some fixed $a, b$ with $a + b > 0$, every positive edge must have signature $\{a,b\}$ and every negative edge must have signature $\{-a, -b\}$. Furthermore, we can assume that all weights take values in $\{0, 1, -1\}$.
\end{Proposition}
\begin{proof}
Take a positive edge $(i,j)$ and a negative edge $(k,\ell)$,
both of which must exist since $\Hc$ decomposes $\Pc_G$. Lemma
\ref{edgeofedgepolytope} and the $4$-cycle condition give that
 $\{i,j\} \cap \{k,\ell\} =
\emptyset$ and the two edges are cycle-compatible with some cycle
$C$.
% say $C = \{\{i,j\},\{j,k\},\{k,\ell\},\{\ell,i\}\}$.
Without loss of generality, say $C = (i,j,k,\ell)$.
Since $s_\Hc((j,k)) = s_\Hc((\ell, i)) = 0$,
one has $a_j = - a_k$ and $a_\ell = -a_i$. So the claim holds true for one pair of edges with opposite sign.
% This shows that the vertices connected to these two edges all take absolute values
% in $\{|a_i|, |a_j|\}$.
% Hence $\{|a_k|, |a_\ell|\} = \{|a_i|, |a_j|\}$.

However, we can say a lot more.
Take any other positive or negative edge $e = (i',j')$ of $G$.
By using Lemma \ref{edgeofedgepolytope} again, we can use a cycle of length $4$
between $e$ and either $(i,j)$ or $(k,\ell)$
(whichever one with the opposite sign as $e$)
in order to show % the absolute values of the vertices in $e$ are again $\{|a_i|, |a_j|\}$.
$\{|a_{i'}|, |a_{j'}|\} = \{|a_i|, |a_j|\}$.
Consequently, if $q$ is a vertex belonging to either a positive or negative edge, then $a_q$ can take at most two possible absolute values.

Now, since $G$ is connected, any other vertex is connected to one of
those vertices via a chain of zero edges.
Note that if $(k',\ell')$ is a zero edge,
then $|a_{k'}| = |a_{\ell'}|$.
Thus the other vertices cannot introduce any new absolute values,
and one has at most two absolute values among all $a_i$'s, as desired.

We now show that we can move the $a_i$ without changing $\Pc_G^{(+)}$ and $\Pc_G^{(-)}$ such that there is at most one nonzero absolute value. If not, without loss of generality only $a$ and $b$ exist as absolute values and $a < b$. Note that moving all $a_i$ with absolute value $a$ to $0$ does not change
any of the signs of the edges and thus the decomposition. So we can assume that besides $0$'s, there is only one absolute value $b$.
We can then scale all these vertices with values $\pm b$ to $\pm 1$ without changing the signs.
\end{proof}

Proposition~\ref{Yan:hyperplane} and Lemma~\ref{planeshift} allow us to restrict attention to hyperplanes of the following form:

\begin{equation}\label{hyper}
\Hc:\,\sum_i \pm x_i =0.
\end{equation}

An intuitive restatement is that we can consider a hyperplane as an assignment of a \emph{zero}, \emph{negative}, or \emph{positive charge} to each vertex, corresponding to assigning the value of $0, -1,$ or $1$ respectively to the weight $a_i$. We can then read the signs of the edges off the graph by just adding the charges at the two incident vertices.

\begin{Example}
{\em
We revisit Example~\ref{ex:first}. We can represent the hyperplane $-x_1+x_4-x_5+x_6=0$ by assigning positive charges to $\{4,6\}$ and negative charges to $\{1,5\}$ and zero charges to the other vertices. This makes $(1,2)$ a negative edge and $(3,4)$ a positive edge; the other edges have zero charge.
}
\end{Example}
\begin{center}
\includegraphics{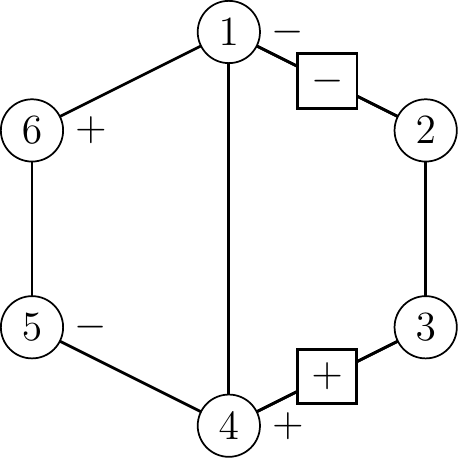}
\end{center}

\begin{Corollary}
\label{decomposition types}
Suppose we have a decomposition of the edge polytope $\Pc_G$. Then we can assume one of the following two cases for the vertices of $G$:
\begin{enumerate}
\item There are no vertices with weight $0$. All positive edges have signature $\{1, 1\}$ and all negative edges have signature $\{-1, -1\}$.
\item There is at least one vertex with weight $0$. All positive edges have signature $\{1, 0\}$ and all negative edges have signature $\{-1, 0\}$. 
\end{enumerate}
\end{Corollary}
\begin{proof}
If there are no vertices with weight $0$, it is easy to see that we are in the first case. It suffices to prove that when we have at least one vertex with weight $0$, the conditions in the second case apply. Since $G$ is connected, there must be at least one edge of signature $\{0, a\}$ where $a \in \{-1, 1\}$. Proposition~\ref{Yan:hyperplane} gives that all positive edges must then be of signature $\{0,1\}$ and all negative edges must have signature $\{0,-1\}$. Thus, no edges of signatures $\{1, 1\}$ or $\{-1, -1\}$ can occur, which is equivalent to saying that all nonzero signatures are $\{1, 0\}$ or $\{-1, 0\}$ (equivalently, the $1$-weighted vertices and the $(-1)$-weighted vertices are isolated sets).
\end{proof}

We call the two types of labelings from Corollary~\ref{decomposition types} \emph{type $I$} and \emph{type $II$} respectively. Using these two types, we give an algorithm to check the decomposability of an edge polytope $P_G$ from its graph:

\begin{Algorithm}
\label{base 2 algorithm}
{\em
We check for type $I$ and type $II$ decomposability separately.

\noindent (\emph{type $I$}) Create an empty list $F$. For every pair of vertex-disjoint edges $e_-$ and $e_+$:
\begin{enumerate}
\item set the signatures of $e_-$ and $e_+$ to $\{-1,-1\}$ and $\{1,1\}$ respectively by setting the weights of relevant vertices.
\item try to set weights $-1$ and $1$ to the other vertices one at a time, each time setting the weight of a vertex adjacent to a vertex with weight already set, until one of the following occurs:
\begin{enumerate}
\item the weights of all vertices are set, in which case we have a decomposition;
\item If we have forced any edge $e$ to be a positive or negative edge, check it against every edge $f$ with opposite sign. We cannot continue if one of the following two things happen:
\begin{enumerate}
\item if $e$ and $f$ are not cycle-compatible ($4$-cycle condition)
\item if $(e,f)$ is in $F$ (we have failed to assign opposite signs to these edges in the past).
\end{enumerate}
if one of these happens, add $(e_-, e_+)$ to $F$ and stop the search.
\end{enumerate}
\end{enumerate}
(\emph{type $II$}) Create an empty list $F$. For every pair of vertex-disjoint edges $e_-$ and $e_+$:
\begin{enumerate}
\item set the signatures of $e_-$ and $e_+$ to $\{-1,0\}$ and $\{1,0\}$ respectively by setting the weights of relevant vertices (we do $4$ for-loops in this case, corresponding to the $4$ possible assignments of the weights).
\item try to set weights $-1$, $0$, or $1$ to the other vertices one at a time, each time setting the weight of a vertex adjacent to a vertex already set with weight $-1$ or $1$ (note this means we always have at most $2$ choices, since we cannot put two $1$'s adjacent to each other or two $-1$'s adjacent to each other by Corollary~\ref{decomposition types}), until one of the following occurs:
\begin{enumerate}
\item all vertices are set, in which case we have a decomposition;
\item there are no non-set vertices adjacent to vertices set with $-1$ and $1$, in which case we may set all the unset vertices to weight $0$ and obtain a decomposition;
\item we prune in a similar manner to the type $I$ case (for every new weighted edge, check if it is compatible with the existing edges of the other parity via checking both the $4$-cycle condition and $F$). If we cannot continue, add $(e_-, e_+)$ to $F$ and stop the search.
\end{enumerate}
\end{enumerate}
}
\end{Algorithm}

Naively, Proposition~\ref{Yan:hyperplane} tells us we can check decomposability by iterating over $3^n$ cases. Algorithm~\ref{base 2 algorithm} cuts the base of the exponent to $2$ in each of the two types. Though the result is still exponential in the worst possible cases (namely very dense graphs), the pruning process should usually provide large optimizations. Note that checking the compatibility conditions of edges with opposite parity is not a bottleneck of the calculation, as in both types we can keep a precomputed hash for all cycle-compatible edge pairs with a one-time $O(n^4)$ calculation.

\begin{Question}
\label{question:complexity}
{\em
Is there a polynomial-time algorithm to decide if $\Pc_G$ is decomposable? Is there a polynomial-time algorithm to decide if $\Pc_G$ is type $I$ (or type $II$) decomposable?
}
\end{Question}

\section{Counting Decompositions}

Given a graph $G$, we may also be interested in counting the number of
decompositions. Since even the decidability of the existence of a decomposition seems to be difficult, we do not expect this to be a tractable problem except in specific cases; moreover, we need to be careful since different separating hyperplanes may give the same decomposition.

\begin{Example}
{\em
Consider the four-cycle $C_4$ with vertices $1,2,3,4$. Then the hyperplanes $x_1-x_4=0$, $x_2-x_3=0$ and $x_1+x_2-x_3-x_4=0$ give us the same decomposition of $\Pc_{C_4}$. Notice that the first two hyperplanes are of type $II$ (with at least one zero coefficient) and the third hyperplane is of type $I$ (with no zero coefficient).
}
\end{Example}

If we restrict to type $I$ hyperplanes, however, we get a unique decomposition (up to sign):
\begin{Lemma}\label{unique}Two different type $I$ hyperplanes will result in different decompositions of the edge polytope, unless they differ only by an overall sign on the weights.
\end{Lemma}
\begin{proof}Suppose two hyperplanes give the same decomposition $\Pc_{G}=\Pc_{G_+}\cup\Pc_{G_-}$. Then up to sign, we can suppose that the edges in $G_0=G_+\cap G_-$ are all zero, the edges in $G_+\backslash G_0$ are all positive, and the edges in $G_-\backslash G_0$ are all negative. Since both hyperplances are of type $I$, we have $a_i=a_j=1$ for the coefficients in the hyperplane for each positive edge $(i,j)$ and $a_k=a_{\ell}=-1$ for each negative edge $(k,\ell)$. Then all the other coefficients are uniquely determined.
\end{proof}

Using this fact, we now prove a couple of auxiliary results that combine to count the number of decompositions for complete multipartite graphs.

\begin{Lemma}
\label{lem:no neighbors}
In a type $II$ decomposition, no vertex with weight $1$ may share a neighbor with a vertex of weight $-1$.
\end{Lemma}
\begin{proof}
By Corollary~\ref{decomposition types}, such a neighbor must be $0$. However, we now have a positive edge next to a negative one, a contradiction.
\end{proof}

\begin{Proposition}
\label{prop:extension} For a connected bipartite graph $G$ with a type $II$ separating hyperplane, there exists a type $I$ separating hyperplane which gives the same decomposition.
\end{Proposition}
\begin{proof}
Let the bipartition of the vertices be $L \cup R$. By Corollary~\ref{decomposition types}, positive edges must have signature $\{1, 0\}$ and negative edges must have signature $\{0, -1\}$. Without loss of generality there is one positive edge $(i,j)$ with $i \in L$ and $j \in R$, such that $a_i = 1$ and $a_j = 0$. Then by Corollary~\ref{opp}, all negative edges $(k, \ell)$ with $a_k = 0$ and $a_\ell = -1$ must have $k \in L$ and $\ell \in R$. By the same reasoning all positive edges must have their weight-$1$ vertex in $L$ and weight-$0$ vertex in $R$. 

At this point, we have covered all vertices involved in nonzero edges, so only zero edges remain. Consider a vertex $i$ with weight $1$ that we have not yet considered. If it exists, since $G$ is connected, there must exist a path starting at $i$ with weights $1, -1, 1, -1, \ldots$ until we get to a vertex next to a nonzero edge. It is then clear that $i$ must be in $L$. Using this and a symmetric argument repeatedly, we see that all vertices with weight $1$ must be in $L$ and all vertices with weight $-1$ must be in $R$. 

Finally, consider the following transformation: change all zero weights in $R$ to $1$ and all zero weights in $L$ to $-1$. We claim that no edge will change sign: any positive edge (with signature $\{1, 0\}$) will change to signature $\{1, 1\}$,  any negative edge (with signature $\{0, -1\}$) will change to signature $\{-1,-1\}$, any zero edge with signature $\{0, 0\}$ will change to $\{-1, 1\}$, and any zero edge with signature $\{1, -1\}$ is unchanged. The result is then a type $II$ hyperplane which gives the same decomposition. 
\end{proof}

\begin{Theorem}\label{commul}Let $G$ be a complete multipartite graph on vertices
$[d]=V_1\cup \dots\cup V_k$ with $k \geq 2$. Then the number of decompositions is
$$ 2^{d-1} - \sum_i (2^{|V_i|} -1) -1. $$
\end{Theorem}
\begin{proof}

First, we show that we only need to consider type $I$ hyperplanes. The proof is different for the cases $k = 2$ and $k > 2$: when $k = 2$, we appeal to Proposition~\ref{prop:extension}. When $k > 2$, every two vertices share a neighbor, so we use Lemma~\ref{lem:no neighbors}. 

There are $2^d$ ways of assigning $1$ or $-1$ to all the vertices. Each of these assignments gives a decomposition unless:
\begin{enumerate}
\item all vertices have the same weight, or
\item both weights occur, but all vertices of one of the weights are inside a single $V_i$ (in which case we do not have an edge of the corresponding sign). 
\end{enumerate} 
There are $2$ assignments for the first case, and $2(2^{|V_i|} - 1)$ ways of confining all weights of some sign (and having at least one) into $V_i$. Summing and dividing by $2$ (using Lemma~\ref{unique}) gives the desired answer.

\end{proof}

In the special case of a complete graph (i.e.  $|V_i| = 1$ for all $i$), this reduces to $2^{d-1} - d - 1$. We remark that this recovers  a special case of Theorem 5.3 in \cite{HJ} when $k=2$, which studies the splitting of hypersimplices $\Delta(k,n)$.

\section{Normal edge polytopes}

Recall that an integral convex polytope $\Pc \subset \RR^d$ is
{\em normal} if, for all positive integers $N$ and for all
$\beta \in N \Pc \cap \ZZ^d$, where $N \Pc = \{ N \alpha : \alpha
\in \Pc \}$, there exists $\beta_1, \ldots, \beta_N$ belonging to
$\Pc \cap \ZZ^d$ such that $\beta = \sum_i \beta_i$.

Now, consider the following condition on $G$: If $C$ and $C'$ are cycles of $G$ of odd length, then either they
share a vertex or there is an edge $(i, j)$ of $G$ such that $i \in C$
and $j \in C'$ (such an edge is
called a {\em bridge} between $C$ and $C'$). This condition is called the {\em odd cycle
condition}, which was first investigated by \cite{FHM} in classical
combinatorics. Its relevance to our situation is the following result, which changes the algebraic condition into a combinatorial condition:

\begin{Proposition}[{Corollary 2.3 of \cite{OHnormal}}]
\label{oddcyclecondition}
$\Pc_G$ is normal if and only if $G$ satisfies the odd cycle condition.
\end{Proposition}

With this result, we now give the main theorem of this section, showing that normality is a hereditary condition on edge polytopes under decomposition:

\begin{Theorem}\label{Yan:Normality} Let $G$ be a finite connected simple graph on
$[d]$ and suppose that the edge polytope decomposes as $\Pc_G = \Pc_{G_+} \cup \Pc_{G_-}$. Then $\Pc_G$ is normal if and only if both
$\Pc_{G_+}$ and $\Pc_{G_-}$ are normal.
\end{Theorem}

\begin{proof}
Recall that $G_+$ has all the zero and positive edges and $G_-$ has all
the zero and negative edges.

Suppose that both $\Pc_{G_+}$ and $\Pc_{G_-}$ are
normal but $\Pc_G$ is not. It follows from the odd cycle
condition that $G$ contains two disjoint odd cycles $C$ and $C'$
with no bridge. Without loss of generality, we may assume that these are minimal in the sense that we cannot pick a smaller odd cycle to replace $C$ while keeping $C'$ fixed such that they are still disjoint and bridgeless (and vice-versa).
In particular, it follows that neither cycle contains a chord, since then a smaller odd cycle could have been selected.
If either $G_+$ or $G_-$ contains all the edges in
both cycles, then $\Pc_G$ cannot be normal. Thus there is at least one
positive and at least one negative edge in the edges of $C$ and $C'$. If they
are in the same cycle $C$, then since the two edges induce a cycle
of length $4$, $C$ must have a chord, a contradiction.
If they are in different cycles, then a cycle of
length $4$ between them introduces two bridges that are zero edges,
which is also a contradiction. Thus $\Pc_G$ is normal.

Conversely, suppose that $\Pc_G$ is normal and that
$\Pc_{G_+}$ is nonnormal. Since $\Pc_G$
is normal, there is at least one bridge between $C$ and $C'$ in $G$. Again by the odd cycle condition, the
subgraph $G_+$ contains two disjoint odd cycles $C$ and $C'$ such
that no bridge between $C$ and $C'$ belongs to $G_+$.
Therefore, all of the bridges between $C$ and $C'$ in $G$ must be negative
edges. We claim all of the edges of $C$ and $C'$ must be zero edges.
In fact, if we had an edge $e$ of either $C$ or $C'$ which is
nonzero, then since $e$ belongs to $G_+$ it must be positive,
and a cycle of length $4$ arises
from $e$ and a (negative) bridge, yielding a zero edge bridge
between $C$ and $C'$, a contradiction.

Now, let $a_1 x_1 + \cdots + a_d x_d = 0$ define the separating hyperplane $\Hc$.  Since all of the edges of
$C$ and $C'$ must be zero edges, one has $a_i + a_j = 0$ if $( i, j
)$ is an edge of either $C$ or $C'$.  However, since both $C$ and
$C'$ are odd, it follows that $a_i = 0$ for all vertices $i$ of $C$
and $C'$. Hence there exists no negative bridge between $C$ and
$C'$, which gives a contradiction.  Hence $\Pc_{G_+}$ is normal.
Similarly, $\Pc_{G_-}$ is normal.
\end{proof}

\begin{Example}
\label{Nan:example}
{\em
When $\Pc_G$ is not normal, different types of decompositions can occur, as shown in the following examples:
\begin{enumerate}
\item It is possible for a nonnormal edge polytope to be decomposed into two nonnormal polytopes. Let $G$ be the finite graph with the following decomposition:

\begin{center}

\includegraphics{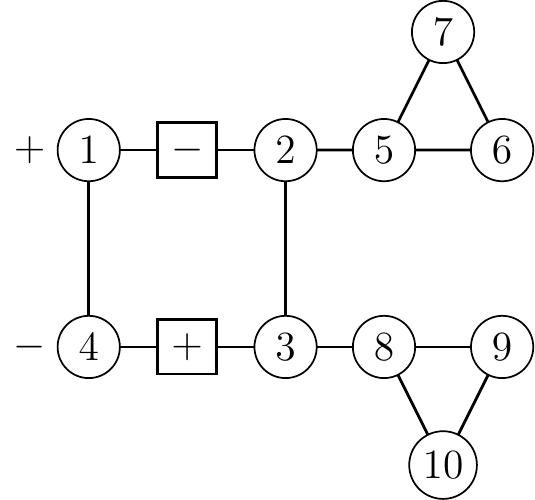}
\end{center}
$\Pc_G$ is not normal because of the two 3-cycles $C_1=(5,6,7)$ and $C_2=(8,9,10)$.
Both subpolytopes $\Pc_{G_+}$ and $\Pc_{G_-}$ are also nonnormal.

\item It is possible for a nonnormal edge polytope to be decomposed into one normal polytope and one nonnormal
polytope. Let $G$ be the finite graph with the following decomposition:

\begin{center}
\includegraphics{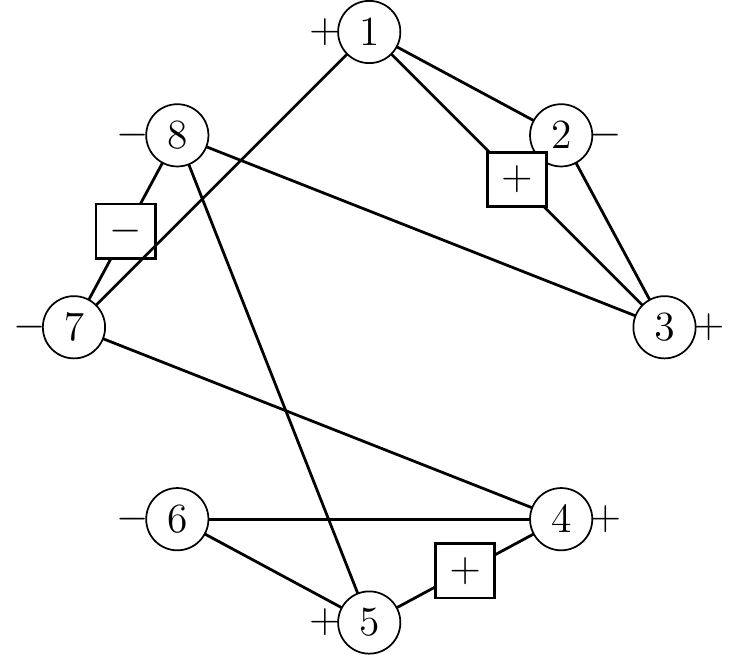}
\end{center}
$\Pc_{G_-}$ is normal and $\Pc_{G_+}$ is nonnormal.

\end{enumerate}

}
% Let $\Hc \subset \RR^{8}$ be the hyperplane defined by  the equation
% \[
% x_1-x_2+x_3+x_4+x_5-x_6-x_7-x_8=0.
% \]
% This decomposes $\Pc_G$ into $\Pc_{G_+}$ and $\Pc_{G_-}$, where
% $E(G_+) = E(G) \setminus \{(7,8)\}$ and $E(G_-) = E(G) \setminus
% \{(4, 6),(1,3) \}$.

\end{Example}

\section{Quadratic toric ideals}
Fix a graph $G$. Let $K[\textbf{t}]=K[t_1,\dots,t_d]$ and $K[\textbf{x}]=K[x_1,\dots,x_n]$ be the polynomial rings over a field $K$ in $d$ and $n$ variables respectively. For each edge $e=\{i,j\} \in E(G)$, we write $\textbf{t}^{e}$ for the squarefree quadratic monomial $t_it_j\in K[\textbf{t}]$, and define the \emph{edge ring} $K[G]$ of $G$ to be the subalgebra of $K[\textbf{t}]$ generated by $\textbf{t}^{e_1}, \textbf{t}^{e_2}, \dots, \textbf{t}^{e_n}$ over $K$. Now define the surjective homomorphism of semigroup rings $\pi:\,K[\textbf{x}]\rightarrow K[G]$ by $\pi(x_i)=\textbf{t}^{e_i}$ for all $1\le i\le n$. The kernel of $\pi$ is called the \emph{toric ideal} of $G$ and denoted by $I_G$. We call a graph $G$ (or its edge polytope $\Pc_G$) \emph{\good{}} if its
toric ideal $I_G$ is generated by quadratic binomials. In the previous section, we studied how the normality property behaves under decomposition; in this section, we do the same with the \good{} property.

Given an even cycle $C = (1,2,\dots,k)$:
\begin{enumerate}
\item call $(i,j)$ an \emph{even (odd) chord} of $C$ if $(j-i) \neq 0$ is odd (even). The naming convention corresponds to the parities of the two subcycles created by $(i,j)$ inside $C$. For example, if $j - i = 2$, then we create a triangle, which is odd. The fact that $C$ is even makes this notion well-defined, as the two subcycles will have the same parity.
\item call $C$ \emph{long} if it has length at least $6$.
\item given two chords $e_1$ and $e_2$ of $C$, say that they \emph{cross} if they contain $4$ distinct vertices that, when we place them in order around $C$, gives the sequence $(a_1, a_2, a_3, a_4)$ where the set of edges $\{e_1, e_2\}$ is the set $\{(a_1, a_3), (a_2, a_4)\}$.
\item call a triple of chords $S = (c_1, c_2, c_3)$ of $C$ an \emph{odd-triple} if at least two of them \emph{cross}.
\end{enumerate}

As before, our strategy is to convert the criterion of being \good{} to that of a combinatorial criterion (Theorem~\ref{howgood}) on graphs, which we will use to prove our final main result, Theorem~\ref{good if both are good}.

\begin{Theorem}[{Theorem 1.2 of \cite{OHquadratic}}]\label{howgood}$G$ is \good{} if and only if the following two conditions hold:
\begin{enumerate}
\item For every even long cycle, there exists either an even chord
or an odd-triple.
\item The induced graph of any two odd-cycles are $2$-connected by bridges (namely, for any two odd cycles which share exactly one node, there
exists a bridge which does not go through the common node, and for any two odd cycles which do not share any nodes, there
exist two bridges).
\end{enumerate}

\end{Theorem}

%\com{Nan, I wonder if we can combine the second two conditions to
%saying that the induced graph on the two odd-cyles are two-connected:
% {\tt en.wikipedia.org/wiki/K-vertex-connected\_graph} . Even if this
% is not helpful we should use existing language in case this makes it easier for other people.(\textbf{nan} can we leave this to Hibi?)}

\begin{Lemma}
\label{even cycle lonely pair} Let $C$ be an even cycle inside a
\good{} $G$, with edges $e_1, e_2, \ldots, e_{2k}$. Suppose $C$
contains exactly one positive $e_i$ and one negative edge $e_j$.
Then $j \cong i \pmod{2}$.
\end{Lemma}
\begin{proof}
Suppose not. Without loss of generality $e_1$ is positive and
$e_{2l}$ is negative. We can
assume that $C$ has vertices $c_1, \ldots, c_{2k}$ with $e_i = (c_i,
c_{i+1})$, where it is understood the vertices are labeled modulo $2k$,
and that either $c_1$ or $c_2$ is $1$. Without loss of generality
$c_2 = 1$. Since the edges $e_2, \ldots, e_{2l-1}$ are all zero
edges, we must have $c_3 = -1, c_4 = 1, c_5 = -1, \ldots, c_{2l} = 1$, which
is a contradiction because $e_{2l} = (c_{2l}, c_{2l+1})$ is negative.
\end{proof}

\begin{Theorem}
\label{good if both are good}
Let $\Pc_G=\Pc_{G_+}\cup \Pc_{G_-}$ be a decomposition. If $\Pc_{G_+}$ and $\Pc_{G_-}$ are both \good{}, then $\Pc_G$ is \good{}. %Assume both $I_{G_+}$
%and $I_{G_-}$ are generated by quadratic binomials, then $I_G$ is
%also generated by quadratic binomials.
\end{Theorem}

\begin{proof}%Suppose  $I_G$ is not generated by quadratic binomials,
Suppose $\Pc_G$ were not \good{}. Then $G$ fails at least one
property listed in Theorem \ref{howgood}. This means one of the following must
be true:
\begin{enumerate}
\item $G$ has a even long cycle $C$ without even chords or odd triples.
\item $G$ has two edge-disjoint odd cycles $C_1$ and $C_2$ that have at most one bridge between them and do not share a vertex.
\item $G$ has two edge-disjoint odd cycles $C_1$ and $C_2$ that have no bridges and share a vertex.
\end{enumerate}

We'll show that each of these cases causes a contradiction:

\begin{enumerate}
\item There must be at least one positive edge $e_1$ in $C$, since otherwise $C$ would appear in $G_-$, failing the assumption that $G_-$ was \good{}. Similarly, $C$ must contain a negative edge $e_2$ as well. There must be a $4$-cycle $(e_1, u, e_2, v)$, so $u$ and $v$ are a pair of chords of $C$ with the same parity. By assumption on $C$ having no even chords, $u$ and $v$ must be odd chords of $C$. Thus, by Lemma~\ref{even cycle lonely pair}, there must be at least one more nonzero edge $e_3$ in $C$. Without loss of generality it is negative; then $e_3$ forms a $4$-cycle with $e_1$ as well and creates another pair of odd chords $C$ by the same reasoning above and one of them must intersect one of $u$ and $v$. However, this creates a odd triple inside $C$, which is a contradiction.

\begin{center}
\includegraphics{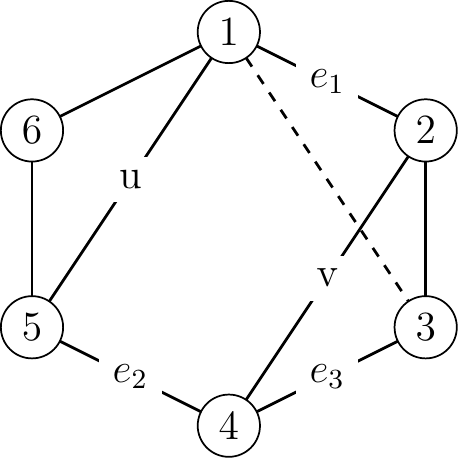}
\end{center}

\item Suppose $d$ is a chord in $C_1$; then $d$ forms a smaller odd cycle with a subset of vertices of $C_1$ that must still have at most one bridge with $C_2$ and share no vertices. Thus, we can assume both $C_1$ and $C_2$ are chordless.

By the same reasoning as the previous case of the long even cycle,
there must be at least one positive edge $e_1$ (and negative edge
$e_2$) in $C_1$ or $C_2$, else $G_-$ (respectively $G_+$) would not
be \good{}. First, note that $e_1$ and $e_2$ must not be in different
cycles, since this would create a $4$-cycle between them that would
serve as $2$ bridges, creating a contradiciton. Thus, all the
nonzero edges must be in one of the cycles, without loss of
generality $C_1$. But this is impossible, since the $4$-cycle
created by $e_1$ and $e_2$ must introduce two chords in $C_1$, which
we assumed to be chordless. Thus, we again have a contradiction.

\item Again, if $d$ is a chord in $C_1$, then either we get a smaller odd cycle that does not share a vertex or any bridges with $C_2$, or we get a smaller cycle that shares a vertex with $C_2$ and no bridges. Thus, we can again assume both $C_1$ and $C_2$ are chordless. The same reasoning as above gives a contradiction. \qedhere
\end{enumerate}
\end{proof}

\begin{Example}
\label{example:other5} {\em
Knowing now that we cannot have both $G_+$ and $G_-$ be \good{} while $G$ is not \good{}, we end this paper with examples showing that all other possibilities under decomposition can be realized. This shows that we cannot extend the statement of Theorem~\ref{good if both are good} to include the ``only if'' case.

\begin{enumerate}
\item $G$ is \good{}, $G_+$, $G_-$ are \good{}. Note that this is Example~\ref{ex:first} again.
\begin{center}
\includegraphics{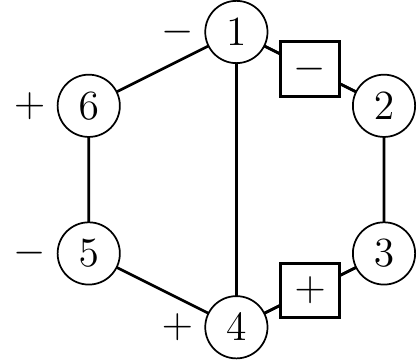}
\end{center}
\item $G$ is \good{}, $G_+$ is not \good{}, and $G_-$ is \good{}.
\begin{center}
\includegraphics{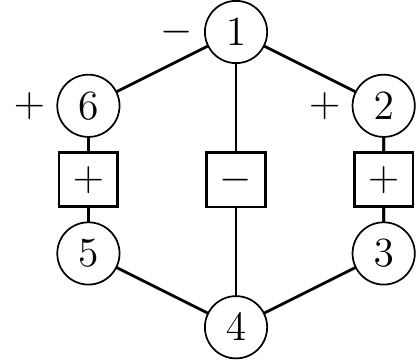}
\end{center}

\item $G$ is \good{}, $G_+$ and $G_-$ are not \good{}.
\begin{center}
\includegraphics{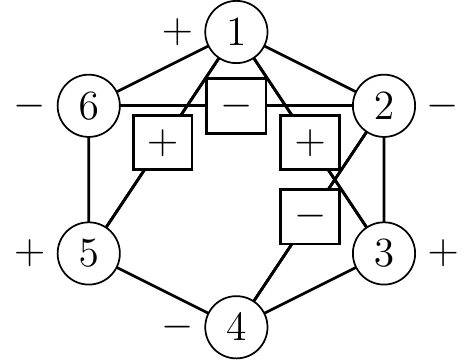}
\end{center}

\item $G$ is not \good{}, $G_+$ and $G_-$ are not \good{}.

\begin{center}
\includegraphics{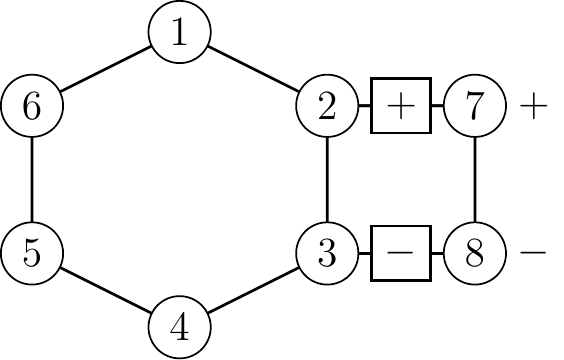}
\end{center}

\item $G$ is not \good{}, $G_+$ is \good{}, and $G_-$ is not \good{}.

\begin{center}
\includegraphics{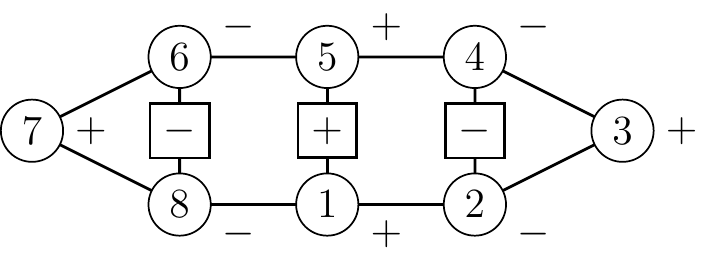}
\end{center}

\end{enumerate}
}
\end{Example}

\end{document}